\documentclass[12pt]{article}

\usepackage{mathtools}
\usepackage{amssymb}
\usepackage{amsthm}
\usepackage{mytheorems}
\usepackage{mymacros}

\theoremstyle{plain}
\newtheorem*{TheoremA}{Theorem A}
\newtheorem*{TheoremB}{Theorem B}
\newtheorem*{TheoremC}{Theorem C}
\newtheorem*{CorollaryD}{Corollary D}

\begin{document}

\title{\bf \boldmath Notes on $K^{\infty}$}

\author{Paul Flavell}

\maketitle

\section{Introduction} \label{intro}

Glauberman's characteristic $p$-functor $K^{\infty}$ \cite{G1,G2}
has remarkable properties.
It controls $p$-transfer in all groups for all primes $p \geq 5$.
It controls $p$-fusion in all $\qd{p}$-free groups for all odd primes.

These notes grew out of the author's attempt to understand and simplify
the definition of $K^{\infty}$.
They are a re-working of some of the ideas in Glauberman's article \cite{G1}.
No new ideas are introduced.
In what follows, all groups are finite.

The definition of $K^{\infty}$ in \cite{G1} is complex.
It involves the construction of two chains of subgroups by means of
an involved group theoretic condition.
We have abstracted and simplified the first part to a general construction
which has little to do with groups.
A similar construction has been employed by Puig \cite[Appendix B]{BG}.
The group theoretic condition has been simplified to a
\emph{no quadratic action} condition reminiscent of that used to define
the Glauberman-Solomon subgroup \cite{F2,GS}.

In \S\ref{dk} we define a characteristic $p$-functor $K^{\infty}$.
We do not know if it yields the same map as Glauberman's definition.
However, that does not matter since we prove that the $K^{\infty}$ we
define enjoys the same properties as Glauberman's.

\begin{TheoremA}
    The map $K^{\infty}$ is a positive characteristic $p$-functor for every prime $p$.
    Moreover the following holds for every $p$-group $S$.
    \begin{itemize}
        \item[(a)]  If $\kinfty{S} \leq T \leq S$ then $\kinfty{S} = \kinfty{T}$.

        \item[(b)]  $\kinfty{S}$ contains every abelian normal subgroup of $S$.

        \item[(c)]  $\cc{S}{\kinfty{S}} \leq \kinfty{S}$.
                    In particular $\zenter{S} \leq \kinfty{S}$.
    \end{itemize}
\end{TheoremA}

\noindent Recall that the first assertion means that $K^{\infty}$ is a map
that sends each $p$-group $S$ to a subgroup $\kinfty{S}$ of $S$
with the properties
\begin{itemize}
    \item   $\kinfty{S\theta} = \kinfty{S}\theta$
            for any isomorphism $\theta$ with domain $S$ and

    \item   $\kinfty{S} \not= 1$ if $S \not= 1$.
\end{itemize}
In particular,
$\kinfty{S}$ is a characteristic subgroup of $S$.

\begin{TheoremB}
    Let $p \geq 5$ be a prime.
    Then $K^{\infty}$ controls $p$-transfer in all groups.
\end{TheoremB}

\noindent Recall that the group has characteristic $p$ if $\cc{G}{\oo{p}{G}} \leq \oo{p}{G}$.
To say that $g$ acts quadratically on $V$ means $[V,g,g] = 1$.

\begin{TheoremC}
    Let $p \geq 3$ be a prime and $G$ a group with characteristic $p$.
    Let $S$ be a subgroup of $G$ with $\oo{p}{G} \leq S$.
    Then
    \begin{itemize}
        \item[(a)]  $\kinfty{S} \normal G$ or

        \item[(b)]  there exists $g \in S \setminus \oo{p}{G}$ and
                    a $G$-chief factor of $\oo{p}{G}$ on which
                    $g$ acts nontrivially and quadratically.
    \end{itemize}
\end{TheoremC}

\noindent The group $\qd{p}$ is the semidirect product of $\spl{2}{p}$
with its natural module.
To say that the group $G$ is $\qd{p}$-free means it is not possible to find
$K \normal H \leq G$ with $H/K \isom \qd{p}$.

\begin{CorollaryD}
    Let $p \geq 3$ be a prime and $G$ a $\qd{p}$-free group with characteristic $p$.
    Let $S$ be a $p$-subgroup of $G$ with $\oo{p}{G} \leq S$.
    Then \[
        \kinfty{S} \normal G.
    \]
\end{CorollaryD}

Results similar to Theorem~C and Corollary~D hold for other characteristic $p$-functors.
Notably Glauberman's $ZJ$-Theorem \cite{G3} and
the Glauberman-Solomon Theorem \cite{GS}.
However, they require $S$ to be a Sylow $p$-subgroup of $G$.
For some applications,
this is too restrictive,
see for example \cite{F1} and \cite{GLS}.

Theorem~A follows readily from the definition of $K^{\infty}$ and
general properties of its construction.
Given some heavy material related to transfer,
which is summarised in \S\ref{t},
Theorem~B follows quickly.
By contrast, the proof of Theorem~C is self contained apart from standard material.

\section{Chains of subgroups} \label{cs}
We construct characteristic functors $K^{\infty}$ with the property:
\[
    \textit{whenever $\kinfty{S} \leq T \leq S$ then $\kinfty{S} = \kinfty{T}$}.
\]

\begin{Notation} \label{cs.1}
    For a group $S$, let $\sgp{S}$ denote the set of subgroups of $S$.
    If $\Sigma \subseteq \sgp{S}$ then $A \leq\in \Sigma$ means that $A$
    is a subgroup of a member of $\Sigma$.
\end{Notation}

\begin{Hypothesis} \label{cs.2}
    The map $\calK(-,-)$ assigns to
    each group $S$ and $\Sigma \subseteq \sgp{S}$
    a subset \[
        \calK(S,\Sigma) \subseteq \sgp{S}
    \]
    and the following hold:
    \begin{itemize}
        \item[(a)]  If $\Delta \subseteq \Sigma$ then \[
                        \calK(S,\Delta) \supseteq \calK(S,\Sigma).
                    \]

        \item[(b)]  If $T \leq S$ and $\Sigma \subseteq \sgp{T}$ then \[
                        \calK(T,\Sigma) = \sgp{T} \cap \calK(S,\Sigma).
                    \]
    \end{itemize}
\end{Hypothesis}

\begin{Example} \label{cs.3}
    \begin{multline*}
        \calK(S,\Sigma) = \\ \set{ A \in \sgp{S} }{ %
            \text{$A$ is abelian and is normalized by each member of $\Sigma$} }.
    \end{multline*}
\end{Example}

\begin{Definition} \label{cs.4}
    Assume Hypothesis~\ref{cs.2} and let $S$ be a group.
    \begin{itemize}
        \item[(a)]  The sequence $\left( \calK_{i}(S) \right)_{i=0}^{\infty}$
                    of subsets of $\sgp{S}$ is defined by \[
                        \text{$\calK_{0}(S) = \sgp{S}$ \;and\; %
                            $\calK_{i+1}(S) = \calK(S,\calK_{i}(S))$ %
                            \, for all $i \geq 0$.}
                    \]

        \item[(b)]  $\calK^{\infty}(S) = \bigcap_{\text{$i$ even}} \calK_{i}(S)$.

        \item[(c)]  $\calK_{\infty}(S) = \bigcup_{\text{$i$ odd}} \calK_{i}(S)$.

        \item[(d)]  $\kinfty{S} = \listgen{\, \calK^{\infty}(S) \, }$,
                    the subgroup generated by the members of $\calK^{\infty}(S)$.
    \end{itemize}
\end{Definition}

\noindent Observe that if $\calK(S,\Sigma)\alpha = \calK(S,\Sigma\alpha)$
for each $\alpha \in\aut{S}$ and $\Sigma \subseteq \sgp{S}$
then each $\calK_{i}(S)$ is $\aut{S}$-invariant and
$\kinfty{S}$ is a characteristic subgroup of $S$.

\newpage

Throughout the remainder of this section we assume that \[
    \text{$\calK(-,-)$ satisfies Hypothesis~\ref{cs.2} and that $S$ is a group.}
\]

\begin{Lemma} \label{cs.5}
    The following hold.
    \begin{itemize}
        \item[(a)]  $\calK_{0}(S) \supseteq \calK_{2}(S) \supseteq \cdots \supseteq \calK^{\infty}(S) %
                \supseteq \calK_{\infty}(S) \supseteq \cdots \supseteq \calK_{3}(S) \supseteq \calK_{1}(S)$.

        \item[(b)]  $\calK(S,\calK^{\infty}(S)) = \calK_{\infty}(S)$ and %
                    $\calK(S,\calK_{\infty}(S)) = \calK^{\infty}(S)$.
    \end{itemize}
\end{Lemma}
\begin{proof}
    (a). Let $k,l \geq 0$.
    Now $\calK_{0}(S) = \sgp{S} \supseteq \calK_{k}(S)$.
    Applying $\calK(S, -)$ to this $2l$ times gives \[
        \calK_{2l}(S) \supseteq \calK_{k+2l}(S).
    \]
    With $k= 2$ this establishes (a) up to the $\calK^{\infty}(S)$ term.
    Applying $\calK(S, -)$ once more gives \[
        \calK_{2l+1}(S) \subseteq \calK_{k+2l+1}(S).
    \]
    Then (a) holds from the $\calK_{\infty}(S)$ term.
    Moreover $\calK_{2l+1}(S) \subseteq \calK_{j}(S)$ for all even $j > 2l+1$
    whence $\calK_{2l+1} \subseteq \calK^{\infty}(S)$ and then
    $\calK_{\infty}(S) \subseteq \calK^{\infty}(S)$.

    (b). Since $S$ is finite there exists $N$ such that
    $\calK^{\infty}(S) = \calK_{i}(S)$ and $\calK_{\infty}(S) = \calK_{j}(S)$
    for all even $i \geq N$ and odd $j \geq N$.
    Let $i \geq N$ be even.
    Then $\calK(S,\calK^{\infty}(S)) = \calK(S,\calK_{i}(S)) = \calK_{i+1}(S) = \calK_{\infty}(S)$.
    Similarly for $\calK_{\infty}(S)$.
\end{proof}

\begin{Lemma} \label{cs.6}
    Let $T \leq S$ and suppose that \[
        \calK^{\infty}(S) \subseteq \sgp{T} \qtext{or} \kinfty{S} \leq T.
    \]
    Then $\calK^{\infty}(S) = \calK^{\infty}(T)$ and $\kinfty{S} = \kinfty{T}$.
\end{Lemma}
\begin{proof}
    In both cases $\calK^{\infty}(S) \subseteq \sgp{T}$.
    The second conclusion follows from the first.
    From the definition of $\calK^{\infty}$ it suffices to show that
    \begin{equation} \tag{$*$}
        \calK^{\infty}(S) \subseteq \calK_{i}(T) \subseteq \calK_{i}(S)
    \end{equation}
    for all even $i \geq 0$.
    Now $\calK_{0}(T) = \sgp{T}$ and $\calK_{0}(S) = \sgp{S}$
    so $(*)$ holds when $i=0$ by hypothesis.

    Let $i \geq 0$ and assume that $(*)$ holds.
    Applying $\calK(S,-)$ and Lemma~\ref{cs.5}(b) we obtain
    \begin{equation} \tag{$**$}
        \calK_{\infty}(S) \supseteq \calK(S,\calK_{i}(T)) \supseteq \calK_{i+1}(S).
    \end{equation}
    The definition of $\calK_{i+1}(T)$ and Hypothesis~\ref{cs.2}(b) yield \[
        \calK_{i+1}(T) = \calK(T,\calK_{i}(T)) = \sgp{T} \cap \calK(S,\calK_{i}(T)).
    \]
    Now $\calK_{\infty}(S) \subseteq \calK^{\infty}(S) \subseteq \sgp{T}$ so $(**)$ becomes \[
        \calK_{\infty}(S) \supseteq \calK_{i+1}(T) \supseteq \calK_{i+1}(S).
    \]
    A further application of $\calK(S,-)$ gives \[
        \calK^{\infty}(S) \subseteq \calK(S,\calK_{i+1}(T)) \subseteq \calK_{i+2}(S).
    \]
    Again we have \[
        \calK_{i+2}(T) = \calK(T,\calK_{i+1}(T)) = \sgp{T} \cap \calK(S,\calK_{i+1}(T)).
    \]
    Since $\calK^{\infty}(S) \subseteq \sgp{T}$ we obtain \[
        \calK^{\infty}(S) \subseteq \calK_{i+2}(T) \subseteq \calK_{i+2}(S).
    \]
    By induction, $(*)$ holds for all $i \geq 0$ and the proof is complete.
\end{proof}

\section{The definition of $K^{\infty}$} \label{dk}

For the remainder of these notes we let $p$ be a prime and
define a map $\calK(-,-)$ as follows: \[
    \begin{minipage}[c]{0.8\textwidth} \em
        for each $p$-group $S$ and $\Sigma \subseteq \sgp{S}$
        then $\calK(S,\Sigma)$ consists of those
        $A \in \sgp{S}$ that satisfy:
        \begin{itemize}
            \item[(a)]  $A$ has nilpotency class at most two,

            \item[(b)]  $A$ is normalized by each element of $\Sigma$ and

            \item[(c)]  whenever $B \leq\in \Sigma$ satisfies $[A,B,B] \leq \zenter{A}$
                        then $[A,B] \leq \zenter{A}$.
        \end{itemize}
    \end{minipage}
\]

\noindent Recall that (a) means $[A,A] \leq \zenter{A}$.
In particular this holds if $A$ is abelian.
Condition (c) may be paraphrased as \[
    \textit{$A/\zenter{A}$ admits no nontrivial quadratic action from $\Sigma$}
\]
and is thus reminiscent of the definition of the Glauberman-Solomon subgroup
as given in \cite{F2}.

Hypothesis~\ref{cs.2} is satisfied and we have the notation \[
     \calK_{i}(S), \quad \calK^{\infty}(S) %
    \qtext{and} \kinfty{S}
\]
as defined in Definition~\ref{cs.4}

\begin{proof}[Proof of Theorem~A]
    It follows directly from the definition that $K^{\infty}$
    is a characteristic $p$-functor.
    Lemma~\ref{cs.6} implies (a).
    Let $A$ be an abelian normal subgroup of $S$.
    Since $A = \zenter{A}$, condition (c) in
    the definition of $\calK(S,\Sigma)$ always holds and
    so $A \in \calK_{i}(S)$ for all $i$.
    Hence $A \in \calK^{\infty}(S)$ and $A \leq \kinfty{S}$.
    Suppose now that $A$ is a maximal abelian normal subgroup of $S$.
    Then $\zenter{S} \leq \cc{S}{\kinfty{S}} \leq \cc{S}{A} = A \leq \kinfty{S}$
    (if not let $\BAR{S} = S/A$ and $B = \listgen{A,x}$
    where $1 \not= \BAR{x} \in \BAR{\cc{S}{A}} \cap \zenter{\BAR{S}}$.)
    Finally if $S \not=1$ then $\zenter{S} \not= 1$ so $K^{\infty}$ is positive.
\end{proof}

\section{Standard results} \label{sr}

\begin{Theorem}[Thompson {\cite[5.3.11]{Gor}}] \label{sr.1}
    Let $P$ be a $p$-group.
    Then $P$ contains a characteristic subgroup $A$ with the properties:
    \begin{itemize}
        \item[(a)]  $A/\zenter{A}$ is elementary abelian.
                    In particular, $A$ has nilpotency class at most two.

        \item[(b)]  $[A,P] \leq \zenter{A}$.

        \item[(c)]  $\cc{P}{A} = \zenter{A}$.

        \item[(d)]  The only $p'$-automorphism of $P$ that centralizes
                    $A$ is the identity automorphism.
    \end{itemize}
\end{Theorem}
\noindent The subgroup $A$ in the conclusion is
called a \emph{critical subgroup} of $P$.

\begin{Lemma} \label{sr.2}
    Let $G$ be a group, $p$ a prime and
    $A \normal G$ a $p$-subgroup.
    Let $g \in G$ and suppose that
    $g$ acts trivially on each $G$-chief factor of $A$.
    Then \[
        g \in \oo{p}{G \bmod \cc{G}{A}}.
    \]
    If in addition $G$ has characteristic $p$ and
    $A$ is a critical subgroup of $\oo{p}{G}$
    then $g \in \oo{p}{G}$.
\end{Lemma}
\begin{proof}
    The first assertion is a standard result.
    To prove the second,
    observe that every $p'$-element of $\cc{G}{A}$ centralizes $\oo{p}{G}$.
    Since $G$ has characteristic $p$,
    this implies that $\cc{G}{A}$ is a $p$-group.
    Then $\oo{p}{G \bmod \cc{G}{A}} = \oo{p}{G}$.
\end{proof}

\begin{Lemma} \label{sr.3}
    Let $G$ be a group, $p$ a prime and
     $A \normal G$ a $p$-subgroup.
     Suppose that $g \in L \subnormal G$ and that $g$ acts quadratically on
     each $L$-chief factor of $A$.
     Then $g$ acts quadratically on each $G$-chief factor of $A$.
\end{Lemma}
\begin{proof}
    By induction we may suppose that $L \normal G$.
    Let $V$ be a $G$-chief factor of $A$.
    Regard $V$ as an irreducible $\gf{p}G$-module.
    Let $U$ be an irreducible $\gf{p}L$-submodule of $V$.
    Then $U$ is an $L$-chief factor of $A$.
    Since $L \normal G$ so is each $Ux$ for $x \in G$.
    By hypothesis, $g$ acts quadratically on each $Ux$
    and hence on their sum, which is $V$.
\end{proof}

\section{Transfer} \label{t}

Recall that if $W$ is a characteristic $p$-functor for the prime $p$
then to say that \emph{$W$ controls $p$-transfer in $G$}
means that for $S \in \syl{p}{G}$, \[
    \begin{minipage}[t]{0.8\textwidth} \it
        the largest abelian $p$-quotient of $G$ is
        isomorphic to the largest abelian $p$-quotient
        of $\nn{G}{W(S)}$.
    \end{minipage}
\]

\noindent A $p$-quotient is a quotient that is a $p$-group.

The proof of the following result is highly nontrivial.
An alternative proof may be found in \cite{Suz}

\begin{Theorem}[Glauberman,Scott {\cite[7.4,p.21]{G2}}] \label{t.1}
    Let $p$ be a prime and $W$ be a characteristic $p$-functor on the group $G$.
    Assume that $W$ satisfies the following condition: \[
        \begin{minipage}[t]{0.85\textwidth}
            whenever $G\Star$ is a section of $G$ with characteristic $p$,
            $S\Star \in \syl{p}{G\Star}$ and $W(S\Star) \notnormal G\Star$
            then there exists $g \in S\Star \setminus \oo{p}{G\Star}$ such that
            $[V,g;p-1] = 1$ for every $G\Star$-chief factor $V$ of $\oo{p}{G\Star}$.
        \end{minipage}
    \]
    Then $W$ controls $p$-transfer in $G$.
\end{Theorem}

\noindent The following is useful in establishing the assumption in Theorem~\ref{t.1}.

\begin{Theorem}[Glauberman, {\cite[10.1,p.30]{G2}}] \label{t.2}
    Let $p$ be a prime and let the group $G$ act on the $p$-group $P$.
    Suppose that $A$ is a critical subgroup of $P$,
    $g \in G$ and $i,j > 0$.
    Assume that
    \begin{align*}
        [V,g;i] &= 1 \quad\text{for every $G$-chief factor $V$ of $A$ and} \\
        [V,g;j] &= 1 \quad\text{for every $G$-chief factor $V$ of $\zenter{A}$.}
    \end{align*}
    Then \[
        [V,g; i+j-1] = 1
    \]
    for every $G$-chief factor $V$ of $P$.
\end{Theorem}

\section{The proofs} \label{tp}

In this section we assume the following:
\begin{itemize}
    \item   $p$ is a prime,

    \item   $G$ is a group with characteristic $p$,

    \item   $S$ is a $p$-subgroup of $G$ with $\oo{p}{G} \leq S$ and

    \item   $A$ is a critical subgroup of $\oo{p}{G}$.
\end{itemize}

\begin{Lemma} \label{tp.1}
    Let $Z$ be an abelian normal subgroup of $S$,
    $i > 0$ and $B \in \calK_{i}(S)$.
    Then \[
        [Z,B,B] = 1.
    \]
\end{Lemma}
\begin{proof}
    Since $Z = \zenter{Z}$,
    the definition of $\calK(-,-)$ implies that $Z \in \calK_{i-1}(S)$.
    Now \[
        [B,Z,Z] \leq [Z,Z] = 1
    \]
    and $B \in \calK_{i}(S) = \calK(S,\calK_{i-1}(S))$.
    Part (c) of the definition of $\calK(-,-)$ forces \[
        [B,Z] \leq \zenter{B}.
    \]
    Of course, $[B,Z] = [Z,B]$ so the conclusion follows.
\end{proof}

\begin{Lemma} \label{tp.2}
    At least one of the following holds.
    \begin{itemize}
        \item[(a)]  $A \in \calK_{3}(S)$.

        \item[(b)]  There exists $g \in S \setminus \oo{p}{G}$ such that \[
                        [A,g,g] \leq \zenter{A} \qtext{and} [\zenter{A},g,g] = 1.
                    \]
                    In particular $[V,g,g] = 1$
                    for every $G$-chief factor $V$ of $A$.
    \end{itemize}
\end{Lemma}
\begin{proof}
    Note that $A \normal S$ and $A$ has nilpotency class at most two.
    Suppose that $A \not\in \calK_{3}(S)$.
    Then by the definition of $\calK_{3}(S)$,
    there exists $B \in \calK_{2}(S)$ such that \[
        [A,B,B] \leq \zenter{A} \qtext{but} [A,B] \not\leq \zenter{A}.
    \]
    Now $A$ is a critical subgroup of $\oo{p}{G}$
    so $[A,\oo{p}{G}] \leq \zenter{A}$.
    Hence there exists $g \in B \setminus \oo{p}{G}$.
    Lemma~\ref{tp.1} implies that $[\zenter{A},g,g] = 1$.
    The final assertion of (b) follows on refining $1 \leq \zenter{A} \leq A$
    to a $G$-chief series for $A$.
\end{proof}

\begin{proof}[Proof of Theorem~B]
    In this case, $p \geq 5$.
    By Theorems~\ref{t.1} and \ref{t.2} it suffices to
    assume $\kinfty{S} \notnormal G$ and prove there exists
    $g \in S \setminus \oo{p}{G}$ with \[
        [A,g,g,g] \leq \zenter{A} \qtext{and} [\zenter{A},g,g] = 1.
    \]
    If $A \not\in\calK_{3}(S)$ then this follows from Lemma~\ref{tp.2}.
    Hence we assume that $A \in \calK_{3}(S)$.

    Now $\kinfty{S} \notnormal G$ so Lemma~\ref{cs.6}
    implies that $\calK^{\infty}(S) \nsubseteq \sgp{\oo{p}{G}}$.
    Choose $B \in \calK^{\infty}(S)$ with $B \nleq \oo{p}{G}$.
    By the definition of $\calK^{\infty}(S)$ we have $B \in \calK_{4}(S)$.

    Since $A \in \calK_{3}(S)$,
    the definition of $\calK_{4}(S)$ implies that $A \leq \nn{S}{B}$
    and $B$ has nilpotency class at most two.
    Then \[
        [A,B,B,B] \leq [B,B,B] = 1.
    \]
    Moreover $[\zenter{A},B,B] = 1$ by Lemma~\ref{tp.1}.
    Choose any $g \in B \setminus \oo{p}{G}$.
\end{proof}

\begin{Lemma} \label{tp.3}
    Suppose that $A \in \calK_{3}(S)$ and let $B \in \calK_{4}(S)$.
    Set $L = B\cc{G}{\zenter{A}}$.
    Then $B$ acts quadratically on every $L$-chief factor of $A$.
\end{Lemma}
\begin{proof}
    The definition of $\calK_{4}(S)$ implies that $A \leq \nn{S}{B}$.
    As $A \normal S$ we obtain
    \begin{equation} \tag{$1$}
        [A,B,B] \leq A \cap [B,B] \leq A \cap \zenter{B}
    \end{equation}
    and then
    \begin{equation} \tag{$2$}
        [A,B,B,A] \leq [A,A \cap \zenter{B}] \leq \zenter{A} \cap \zenter{B}.
    \end{equation}
    Let \[
        \widetilde{A} = \zenter{A \bmod \zenter{A} \cap \zenter{B}}.
    \]
    Note that $\widetilde{A} \normal L$ because $A \normal G$ and $\zenter{A} \cap \zenter{B} \normal L$.
    From $(2)$,
    \begin{equation} \tag{$3$}
        [A,B,B] \leq \widetilde{A}.
    \end{equation}
    Now $\widetilde{A} \leq\in \calK_{3}(S)$ and
    $[B, \widetilde{A}, \widetilde{A}] \leq [\widetilde{A},\widetilde{A}]
    \leq \zenter{A} \cap \zenter{B}$.
    Since $B \in \calK_{4}(S)$,
    this forces $[B,\widetilde{A}] \leq \zenter{B}$ whence
    \begin{equation} \tag{$4$}
        [\widetilde{A},B,B] = 1.
    \end{equation}
    As $\widetilde{A} \normal L$ we may refine the series $1 \leq \widetilde{A} \leq A$
    to an $L$-chief series for $A$ and
    the conclusion follows from $(3)$ and $(4)$.
\end{proof}

\begin{proof}[Proof of Theorem~C]
    We assume that conclusion (b) fails.
    Recall that
    $\oo{p}{G}$ acts trivially on each $G$-chief factor of $\oo{p}{G}$.
    Thus for all $G$-chief factors $V$ of $\oo{p}{G}$ and $g \in S$,
    \begin{equation} \tag{$*$}
        [V,g,g] = 1 \qtext{implies} [V,g] = 1.
    \end{equation}
    Note that each $G$-chief factor of $A$ or $\zenter{A}$ is also
    a $G$-chief factor of $\oo{p}{G}$.

    Suppose that $A \not\in \calK_{3}(S)$.
    Choose $g$ in accordance with Lemma~\ref{tp.2}.
    Then $(*)$ and Lemma~\ref{sr.2} force $g \in \oo{p}{G}$,
    a contradiction.
    Thus $A \in \calK_{3}(S)$.

    Let $B \in \calK^{\infty}(S)$.
    Then $B \in \calK_{4}(S)$.
    Lemma~\ref{tp.1} and $(*)$ imply that
    $B$ acts trivially on each $G$-chief factor of $\zenter{A}$.
    Then Lemma~\ref{sr.2} forces \[
        B \leq \oo{p}{G \bmod \cc{G}{\zenter{A}}}.
    \]
    Let $L = B\cc{G}{\zenter{A}}$.
    Then $L \subnormal G$.
    Lemmas~\ref{tp.3} and \ref{sr.3} imply that $B$ acts quadratically
    on every $G$-chief factor of $A$.
    Then $(*)$ and Lemma~\ref{sr.2} force $B \leq \oo{p}{G}$.
    We deduce that $\calK_{\infty}(S) \subseteq \sgp{\oo{p}{G}}$.
    Lemma~\ref{cs.6} yields $\kinfty{S} = \kinfty{\oo{p}{G}} \normal G$
    so conclusion (a) holds.
\end{proof}

Corollary~D follows immediately from Theorem~C and a
theorem of Glauberman \cite[Lemma~6.3]{G3} or \cite[Theorem~B]{F2}.

\bibliographystyle{amsalpha}

\end{document}